\documentclass[11pt]{amsart}

\usepackage[T1]{fontenc}
\usepackage{lmodern}
\usepackage{amsmath,amssymb,mathtools}
\usepackage{graphicx}
\usepackage{tikz}
\usepackage{booktabs}
\usepackage{microtype}
\usepackage{flafter}
\usepackage{placeins}
\usepackage{amsrefs}
\usepackage{hyperref}
\usepackage[nameinlink,capitalize]{cleveref}

\newtheorem{theorem}{Theorem}[section]
\newtheorem{proposition}[theorem]{Proposition}
\newtheorem{lemma}[theorem]{Lemma}
\newtheorem{corollary}[theorem]{Corollary}
\theoremstyle{definition}
\newtheorem{definition}[theorem]{Definition}
\theoremstyle{remark}
\newtheorem{remark}[theorem]{Remark}

\DeclareMathOperator{\ext}{ext}

\newcommand{\R}{\mathbb R}
\newcommand{\OO}{\mathfrak O}
\newcommand{\BO}{\mathsf B_{\OO}}

\newcommand{\Ptwo}{\mathcal P({}^2\R^2)}
\newcommand{\normO}[1]{\lVert #1\rVert_{\OO}}
\newcommand{\bilnorm}[1]{\lVert #1\rVert_{\OO,\mathrm{bil}}}
\newcommand{\Bern}{\mathfrak B_{\OO}}
\newcommand{\Mark}{\mathfrak M_{\OO}}
\newcommand{\Unc}{\mathbf u_{\mathrm{can}}(\OO)}

\usepackage[left=3.33cm,top=3cm,right=3.33cm,bottom=3.0cm]{geometry}

\title[Classical inequalities on an octagonal sector]
{Classical Polynomial Inequalities for Quadratic Forms on an Octagonal Sector}

\author[Han]{Manwook Han}
\address[Manwook Han]{\mbox{}\newline\indent  Department of Mathematics,\newline\indent Chungbuk National University\newline\indent Cheongju, Chungbuk 28644, \newline\indent Republic of Korea}
\email{mwhan0828@gmail.com}

\author[Kim]{Sun Kwang Kim}
\address[Sun Kwang Kim]{\mbox{}\newline\indent  Department of Mathematics,\newline\indent Chungbuk National University\newline\indent Cheongju, Chungbuk 28644, \newline\indent Republic of Korea}
\email{skk@chungbuk.ac.kr}

\author[Seoane]{Juan B. Seoane--Sep\'ulveda}
\address[Juan B. Seoane--Sep\'ulveda]{\mbox{}\newline\indent  Instituto de Matem\'atica Interdisciplinar (IMI) \newline\indent and \newline\indent Departamento de An\'alisis Matem\'atico y Matem\'atica Aplicada, \newline\indent Facultad de Ciencias Matem\'aticas, \newline\indent Universidad Complutense de Madrid, \newline\indent Plaza de Ciencias 3, \newline\indent 28040 Madrid, Spain}
\email{jseoane@ucm.es}

\thanks{Corresponding author: Sun Kwang Kim.}

\subjclass[2020]{Primary 41A17, 26D05; Secondary 46G25, 52A21, 46B04}
\keywords{Bernstein--Markov inequalities, polarization constants,
unconditional constants, coefficient inequalities, quadratic forms}

\begin{document}

\begin{abstract}
We establish a collection of sharp inequalities for real quadratic forms on the
first-quadrant sector of a regular octagon.  Starting from the complete
extreme-point description of the associated polynomial unit ball, we compute the
exact pointwise Bernstein function for the Euclidean gradient.  One extreme curve
controls the problem: its endpoint is active up to slope $1/2$, after which the
maximizer follows an explicit Cardano branch.  We obtain the sharp Markov constant
$2\sqrt5$, the exact relative quadratic polarization constant $2$, the canonical
unconditional constant $3$, and the body-relative Bohr radius $1/\sqrt3$.  We also
determine the optimal coefficient $\ell_q$-comparison for every
$1\le q\le\infty$.  The same norm-one polynomial is extremal for all these global
constants.  At $q=4/3$ the result is a sharp fixed-space coefficient inequality of
\textit{Bohnenblust--Hille type}.
\end{abstract}

\maketitle
\enlargethispage{4pt}

\section{Introduction}\label{sec:introduction}

An exact norm formula in a low-dimensional polynomial space is often the
beginning, rather than the end, of the analysis.  It can reveal the geometry of the
unit sphere, identify the extreme points of the unit ball, and convert sharp
polynomial inequalities into explicit finite-dimensional optimization problems.
Konheim and Rivlin gave an early characterization of extreme polynomials in the
unit ball of real polynomials bounded on an interval \cite{KonheimRivlin}.  For
quadratic polynomials, Aron and Klimek developed a particularly effective program:
compute the norm, determine a planar projection of the unit ball, parametrize the
unit sphere, and then classify its extreme points \cite{AronKlimek}.  This program
has since been implemented in many three-dimensional polynomial spaces; systematic
accounts, further examples, and applications appear in \cites{FGMMRS,GJMMMS}.

The convex mechanism behind the applications is the finite-dimensional
Krein--Milman reduction.  Once the extreme set is known, the maximum of every
continuous convex functional on the polynomial unit ball can be sought there.
Thus an optimization over an entire three-dimensional body is replaced by finitely
many isolated points and one-parameter families.  This method has produced sharp
derivative inequalities, polarization and unconditional constants, and coefficient
estimates for the simplex, the square, circular sectors, spaces of trinomials, and
polygonal norms; see, for example,
\cites{AJMS,AMRS,GMSS,GMSU,JMPS,KimHex,MRS}.  Closely related polynomial and
bilinear geometries for symmetric octagonal norms were studied in
\cites{KimOctBil,KimOctGeom}.

The derivative questions considered here belong to the classical
Bernstein--Markov tradition.  A detailed historical account, beginning with
Mendeleev's quadratic problem and ending with the exact low-degree pointwise
Bernstein functions, is given at the start of \cref{sec:bernstein}.  Multivariate
and infinite-dimensional extensions have been developed for convex bodies and
Banach spaces in
\cites{FSMS,Ganzburg,Harris,Kroo2026,KrooRevesz,MilevRevesz,MuSar,Sarantopoulos,Wilhelmsen}.
The terminology records two different extremal questions: a Bernstein function is
the best derivative bound at a prescribed point, whereas a Markov constant is the
maximum of that pointwise function over the whole norming body.

Polarization supplies a second motivation.  Let $X$ be a real or complex Banach
space, let $P$ be a continuous $n$-homogeneous polynomial on $X$, and let
$\check P$ be its unique symmetric $n$-linear polar.  With both norms taken on the
unit ball of $X$, Martin proved in his 1932 thesis that
\begin{equation}\label{eq:Martin-introduction}
 \|P\|_X\le \|\check P\|_X\le \frac{n^n}{n!}\,\|P\|_X.
\end{equation}
The first inequality follows by evaluating the polar on the diagonal.  The second
is universal---its factor is independent of both the dimension and the Banach
space---and is generally optimal: equality occurs on $\ell_1^n$ for the Nachbin
polynomial $N_n(x_1,\ldots,x_n)=x_1\cdots x_n$
\cites{Martin,GJMMMS}.  In degree two, Martin's factor is $2$.  An unbalanced
convex body is not a Banach-space unit ball, so \eqref{eq:Martin-introduction} does
not apply to it.  Nevertheless, the polynomial and multilinear supremum norms
induced by the same body can still be compared.  The known relative quadratic
constants are $3$ for the simplex and $3/2$ for the square
\cites{MRS,GMSS,GJMMMS}; the octagonal sector studied here realizes the intervening
value $2$ exactly.

Unconditionality leads to a third, genuinely different, coefficient problem.  For
a polynomial written in the canonical monomial basis, one asks how much its norm
can increase when coefficients are independently suppressed or their signs are
changed.  On a body contained in the positive quadrant, this is equivalent to
comparing $P$ with the polynomial obtained by replacing every coefficient by its
absolute value.  This modulus formulation and its relation to sharp polynomial
inequalities were developed in \cite{GMSU}.  Polynomial supremum norms are usually
non-absolute, so the canonical unconditional constant contains geometric
information not captured by the separate coefficient bounds.

The Bohnenblust--Hille theory is related to coefficient estimates but must be kept
conceptually separate.  The classical multilinear inequality asks, for each degree
$m$, for one constant valid simultaneously in every dimension $N$ for all
$m$-linear forms on $(\ell_\infty^N)^m$; equivalently, it is an
infinite-dimensional estimate on $c_0$.  Its degree-two case is Littlewood's
$4/3$ inequality \cite{Littlewood}, whose optimal real multilinear constant is
$\sqrt2$ \cite{DinizEtAl}.  The polynomial Bohnenblust--Hille inequality has the
same uniform-in-$N$ character and the critical exponent $2m/(m+1)$
\cites{BohnenblustHille,DefantEtAl}.  By contrast, the present coefficient space has
fixed dimension three, and its norm is defined on an unbalanced body rather than an
$\ell_\infty$ ball.  We determine the exact equivalence constants for all
coefficient $\ell_q$-norms.  The case $q=4/3$ is therefore described only as an
inequality \emph{of Bohnenblust--Hille type}: it uses the classical quadratic
exponent, but it is neither the original multilinear inequality nor the classical
polynomial Bohnenblust--Hille inequality.

Throughout, a \emph{convex body} in $\R^d$ means a compact convex set with nonempty
interior.  It is \emph{balanced} if $K=-K$.  We work on the unbalanced convex body
\begin{equation}\label{eq:body}
 \OO:=\{(x,y)\in[0,1]^2:x+y\le\sqrt2\}.
\end{equation}
It is the first-quadrant sector of the centrally symmetric regular octagon
\[
 \widetilde{\OO}:=\{(x,y)\in\R^2:|x|\le1,\ |y|\le1,\ |x|+|y|\le\sqrt2\}.
\]
Put
\begin{equation}\label{eq:tau}
 \tau:=\sqrt2-1,
 \qquad \tau^2=1-2\tau,
 \qquad \sigma:=\tau-\tau^2=3\sqrt2-4.
\end{equation}
The non-axial radial boundary of $\OO$ consists of
\begin{align}
 L_1&=\{(x,1):0\le x\le\tau\},\notag\\
 L_2&=\{(x,\sqrt2-x):\tau\le x\le1\},\label{eq:radial-sides}\\
 L_3&=\{(1,y):0\le y\le\tau\}.\notag
\end{align}
Every nonzero homogeneous polynomial attains its supremum norm on
$L_1\cup L_2\cup L_3$.

Let $\Ptwo$ denote the three-dimensional real vector space of quadratic homogeneous
polynomials.  Thus $P\in\Ptwo$ satisfies $P(tx,ty)=t^2P(x,y)$ and has the form
\[
 P(x,y)=ax^2+bxy+cy^2.
\]
We write
\[
 \normO{P}:=\sup_{(x,y)\in\OO}|P(x,y)|,
 \qquad
 \BO:=\{P\in\Ptwo:\normO{P}\le1\}.
\]
The companion paper \cite{HKMSGeometry} gives an exact five-region norm formula,
all vertical sections of $\BO$, a parametrization of the unit sphere, and a complete
classification of the extreme points.  In particular, every continuous convex
functional on $\BO$ can be maximized over four one-parameter curves and four pairs
of isolated polynomials.  The present article turns that geometric information into
sharp inequalities.

Our first main result is the exact pointwise Euclidean Bernstein function
\[
 \Bern(x,y):=\sup_{\normO{P}\le1}\|\nabla P(x,y)\|_2.
\]
After reducing by symmetry to points $(x,rx)$, $0\le r\le1$, we prove that one
extreme curve dominates every other extreme point.  A genuine phase transition
occurs at $r=1/2$: for $r\le1/2$ a fixed endpoint polynomial is extremal, whereas
for $r>1/2$ the maximizing point moves along the curve and is given explicitly by
the trigonometric solution of a cubic equation.  The pointwise formula yields the
sharp Markov inequality
\[
 \max_{(x,y)\in\OO}\|\nabla P(x,y)\|_2
 \le2\sqrt5\,\normO{P}.
\]
We also prove that the exact relative quadratic polarization constant is $2$.

The coefficient conclusions form a second group of exact results.  For every
$1\le q<\infty$,
\[
 \sup_{P\ne0}\frac{(|a|^q+|b|^q+|c|^q)^{1/q}}{\normO{P}}
 =(2+4^q)^{1/q},
\]
while the corresponding $\ell_\infty$ constant is $4$.  At $q=4/3$ this gives a
sharp fixed-space coefficient inequality of Bohnenblust--Hille type, with factor
$(2+2^{8/3})^{3/4}$.  We further determine the unconditional constant of the
canonical monomial basis $(x^2,xy,y^2)$ and obtain the exact value $3$.  The same
polynomial
\begin{equation}\label{eq:Pstar-intro}
 P_*(x,y):=-x^2+4xy-y^2
\end{equation}
attains the Markov, polarization, coefficient, and unconditional constants.  In
particular,
\[
 \Mark=2\sqrt{1+\mathbf c_2(\OO)^2}=2\sqrt5.
\]

The paper is organized as follows.  In \cref{sec:geometry} we give the necessary
convex preliminaries and recall the geometric input.  The exact Bernstein function
is obtained in \cref{sec:bernstein}, and the global Markov constant is determined in
\cref{sec:markov}.  The relative polarization problem is solved in
\cref{sec:polarization}.  Finally, \cref{sec:coefficients} treats the complete scale
of coefficient norms, distinguishes the resulting $4/3$ estimate from the classical
Bohnenblust--Hille inequalities, and determines the canonical unconditional
constant and its associated body-relative Bohr radius.

\begin{figure}[t]
\centering
\begin{tikzpicture}[scale=3.35]
  \pgfmathsetmacro{\tt}{sqrt(2)-1}
  \draw[->,line width=.75pt] (-0.12,0)--(1.17,0) node[right] {$x$};
  \draw[->,line width=.75pt] (0,-0.12)--(0,1.17) node[above] {$y$};
  \draw[gray!55,dashed,thick]
    (1,\tt)--(\tt,1)--(-\tt,1)--(-1,\tt)--(-1,-\tt)--(-\tt,-1)--
    (\tt,-1)--(1,-\tt)--cycle;
  \fill[gray!28] (0,0)--(1,0)--(1,\tt)--(\tt,1)--(0,1)--cycle;
  \draw[line width=1.1pt] (0,0)--(1,0)--(1,\tt)--(\tt,1)--(0,1)--cycle;
  \draw[line width=1.3pt] (0,1)--(\tt,1)--(1,\tt)--(1,0);
  \foreach \p in {(0,0),(1,0),(1,\tt),(\tt,1),(0,1)}
    \fill \p circle (1.05pt);
  \node at (.46,.48) {$\OO$};
  \node[above,font=\scriptsize] at (.20,1) {$L_1$};
  \node[above right,font=\scriptsize] at (.70,.72) {$L_2$};
  \node[right,font=\scriptsize] at (1,.18) {$L_3$};
  \node[above right,font=\scriptsize] at (\tt,1) {$(\tau,1)$};
  \node[right,font=\scriptsize] at (1,\tt) {$(1,\tau)$};
\end{tikzpicture}
\caption{The octagonal sector \(\OO\).  The dashed polygon is the regular octagon
\(\widetilde{\OO}\), and the bold broken line \(L_1\cup L_2\cup L_3\) is the
non-axial radial boundary.}
\label{fig:domain}
\end{figure}
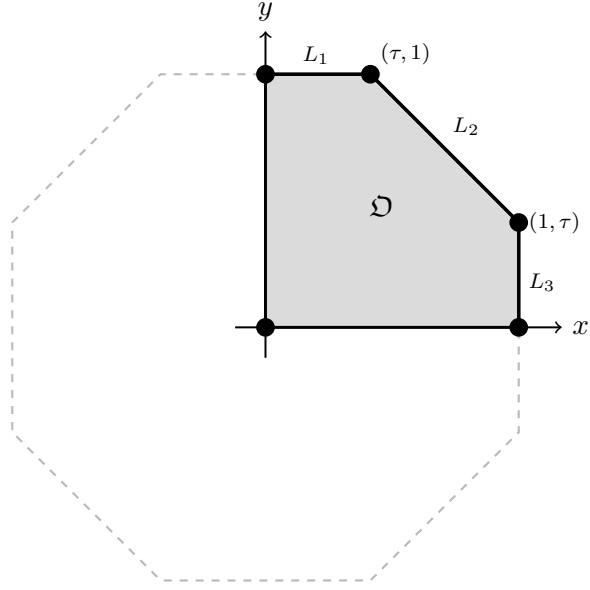

\section{Convex preliminaries and the geometric input}\label{sec:geometry}

A point $e$ of a convex set $C$ is called \emph{extreme} if
$e=\theta x+(1-\theta)y$, with $x,y\in C$ and $0<\theta<1$, implies
$x=y=e$.  We denote the set of all extreme points by $\ext(C)$.  The
finite-dimensional Krein--Milman theorem states that a compact convex set is the
convex hull of its extreme points \cite{KreinMilman}.  The following consequence is
the form used throughout the paper.

\begin{proposition}[Krein--Milman reduction]\label{prop:KM}
Let $C$ be a compact convex subset of a finite-dimensional real vector space, and let
$\Phi:C\to\R$ be continuous and convex.  Then
\[
 \max_{x\in C}\Phi(x)=\max_{e\in\ext(C)}\Phi(e).
\]
\end{proposition}

\begin{proof}
Choose $x_0\in C$ with $\Phi(x_0)=\max_C\Phi$.  By the finite-dimensional
Krein--Milman theorem, $x_0$ is a convex combination of extreme points.  Convexity
then shows that at least one of those extreme points has $\Phi$-value equal to
$\Phi(x_0)$.
\end{proof}

We now recall the part of the unit-ball geometry from \cite{HKMSGeometry} that will
be used below.  The formulas are rewritten in parameters adapted to the derivative
problem.

For \(0\le s\le\tau\) and \(\tau\le u\le1\), put
\begin{align}
 \gamma_1(s)&:=(-1,2s,1-s^2),
 &\gamma_4(s)&:=(1-s^2,2s,-1),\label{eq:gamma14}\\
 \gamma_2(u)&:=\left(-1,1+3u,\frac{1-3u^2}{2}\right),
 &\gamma_3(u)&:=\left(\frac{1-3u^2}{2},1+3u,-1\right).
 \label{eq:gamma23}
\end{align}
The coordinates are ordered as \((a,b,c)\).  These parametrizations are equivalent
to those in the companion paper: for example, on \(\Gamma_2\) one has
\(t=(1-3u^2)/2\) and \(1+\sqrt{3(1-2t)}=1+3u\).

Define also
\begin{equation}\label{eq:isolated}
 p_0=(0,0,1),\quad p_1=(1,0,0),\quad
 p_2=\left(\frac12,1,\frac12\right),\quad p_3=(1,-\tau,1).
\end{equation}

\begin{theorem}[Geometric input; \cite{HKMSGeometry}]\label{thm:geometric-input}
The extreme points of \(\BO\) are
\begin{equation}\label{eq:extreme-set}
 \ext(\BO)=
 \bigcup_{j=1}^4\bigl(\Gamma_j\cup(-\Gamma_j)\bigr)
 \cup\{\pm p_0,\pm p_1,\pm p_2,\pm p_3\},
\end{equation}
where \(\Gamma_j\) is the image of \(\gamma_j\) in
\eqref{eq:gamma14}--\eqref{eq:gamma23}.  Every listed polynomial has norm one.
\end{theorem}

The endpoint \(u=1\) of \(\Gamma_2\) will play a distinguished role below:
\begin{equation}\label{eq:Pstar-geometric}
 \gamma_2(1)=(-1,4,-1),
\end{equation}
which is the coefficient triple of the polynomial \(P_*=-x^2+4xy-y^2\).

Consequently, \cref{prop:KM} gives
\begin{equation}\label{eq:KM-reduction}
 \max_{P\in\BO}\Phi(P)=\max_{P\in\ext(\BO)}\Phi(P)
\end{equation}
for every continuous convex functional $\Phi:\BO\to\R$.  This is the only
geometric result from the companion paper required in the proofs of the sharp
inequalities.

\begin{figure}%[t]
\centering
\includegraphics[width=.96\textwidth]{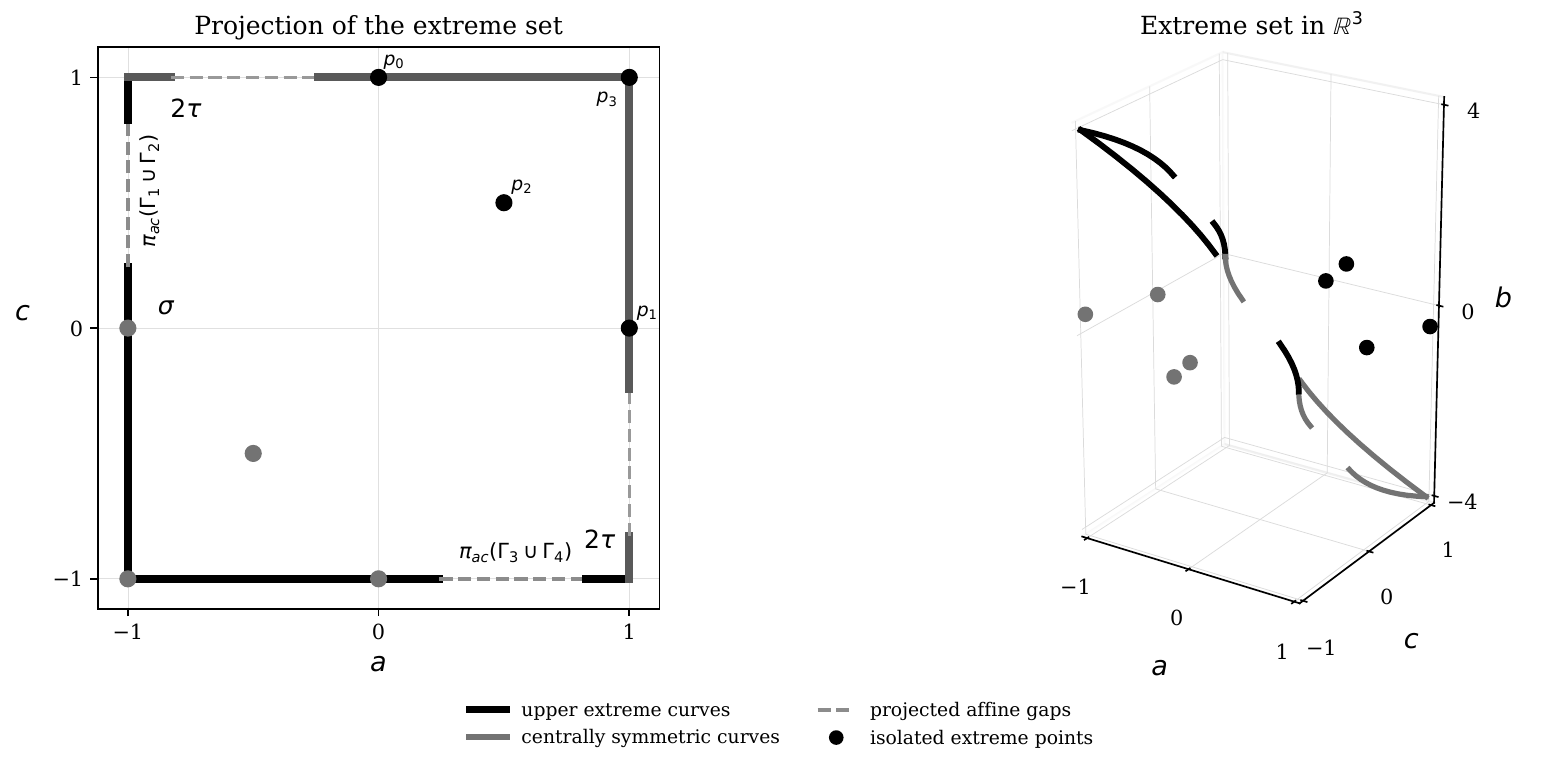}
\caption{The geometric input from \cref{thm:geometric-input}.  The left panel shows
the projection of the extreme set onto the coefficient plane; the right panel shows
the four upper curves and isolated points together with their centrally symmetric
counterparts.}
\label{fig:extreme-geometry}
\end{figure}

\section{The exact Euclidean Bernstein function}\label{sec:bernstein}

\subsection*{Historical background}

The problem of estimating derivatives of a polynomial from its supremum norm has a
remarkably concrete origin. In the late nineteenth century, D.~I. Mendeleev asked
how large the derivative of a quadratic polynomial can be when the polynomial is
bounded on a real interval. After an affine change of variable, the interval may be
normalized to $[-1,1]$. If
\[
\|P\|_{[-1,1]}:=\max_{-1\le t\le1}|P(t)|,
\]
Mendeleev's quadratic estimate asserts that
\[
\|P'\|_{[-1,1]}\le4\|P\|_{[-1,1]}.
\]
The constant $4$ is optimal: for $P(t)=1-2t^2$, equality occurs at $t=\pm1$.
This question, which arose from Mendeleev's work on the specific gravity of
alcoholic solutions, led to the classical investigations of the Markov brothers;
see \cites{Markov1889,FSMS} and the historical accounts in
\cites{BorweinErdelyi,KalmykovNagyTotik}.

A.~A. Markov extended the quadratic result to arbitrary degree in 1889.

\begin{theorem}[A.~A. Markov, 1889]\label{thm:classical-AAMarkov}
	Let $P$ be a real polynomial of degree at most $n$. Then
	\begin{equation}\label{eq:classical-AAMarkov}
		\|P'\|_{[-1,1]}\le n^2\|P\|_{[-1,1]}.
	\end{equation}
	The constant $n^2$ is optimal. Equality is attained at the endpoints by the
	$n$th Chebyshev polynomial of the first kind,
	\[
	T_n(t)=\cos\bigl(n\arccos t\bigr),\qquad -1\le t\le1.
	\]
\end{theorem}

V.~A. Markov subsequently obtained the sharp analogue for every higher derivative;
we cite the German translation of the 1892 Russian original in \cite{VMarkov}.

\begin{theorem}[V.~A. Markov, 1892]\label{thm:classical-VAMarkov}
	Let $P$ be a real polynomial of degree at most $n$, and let $1\le k\le n$.
	Then
	\begin{equation}\label{eq:classical-VAMarkov}
		\|P^{(k)}\|_{[-1,1]}
		\le T_n^{(k)}(1)\,\|P\|_{[-1,1]}
		=\frac{n^2(n^2-1)\cdots\bigl(n^2-(k-1)^2\bigr)}
		{1\cdot3\cdots(2k-1)}\,\|P\|_{[-1,1]}.
	\end{equation}
	The constant is optimal, and equality is attained by $\pm T_n$ at the endpoints.
\end{theorem}

The Markov inequalities are global: their constants control the derivative
simultaneously at every point of the interval. In the interior one can ask the finer
pointwise question. For $t\in[-1,1]$, define
\begin{equation}\label{eq:classical-Bnk-def}
	\mathcal B_{n,k}(t)
	:=\sup\bigl\{|P^{(k)}(t)|:\deg P\le n,
	\ \|P\|_{[-1,1]}\le1\bigr\},
\end{equation}
and write $\mathcal B_n:=\mathcal B_{n,1}$. Thus $\mathcal B_{n,k}(t)$ is the
smallest constant for which
\[
|P^{(k)}(t)|\le \mathcal B_{n,k}(t)\|P\|_{[-1,1]}
\]
holds for every polynomial of degree at most $n$. Bernstein's 1912 theorem gives
the fundamental pointwise estimate \cite{Bernstein1912}.

\begin{theorem}[S.~Bernstein]\label{thm:classical-Bernstein}
	For every real polynomial $P$ of degree at most $n$ and every $t\in(-1,1)$,
	\begin{equation}\label{eq:classical-Bernstein}
		|P'(t)|\le \frac{n}{\sqrt{1-t^2}}\,\|P\|_{[-1,1]}.
	\end{equation}
	Equivalently,
	\[
	\mathcal B_n(t)\le \frac{n}{\sqrt{1-t^2}}.
	\]
\end{theorem}

\begin{figure}%[!htbp]
	\centering
	{%
		\setlength{\tabcolsep}{0pt}
		\begin{tabular*}{\textwidth}
			{@{\extracolsep{\fill}}cccc@{}}
			
			\includegraphics[height=3.75cm]{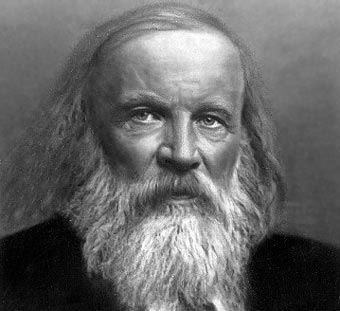}
			&
			\includegraphics[height=3.75cm]{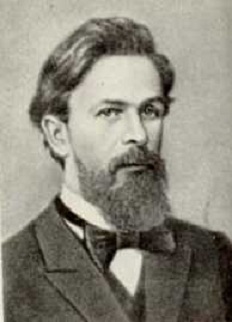}
			&
			\includegraphics[height=3.75cm]{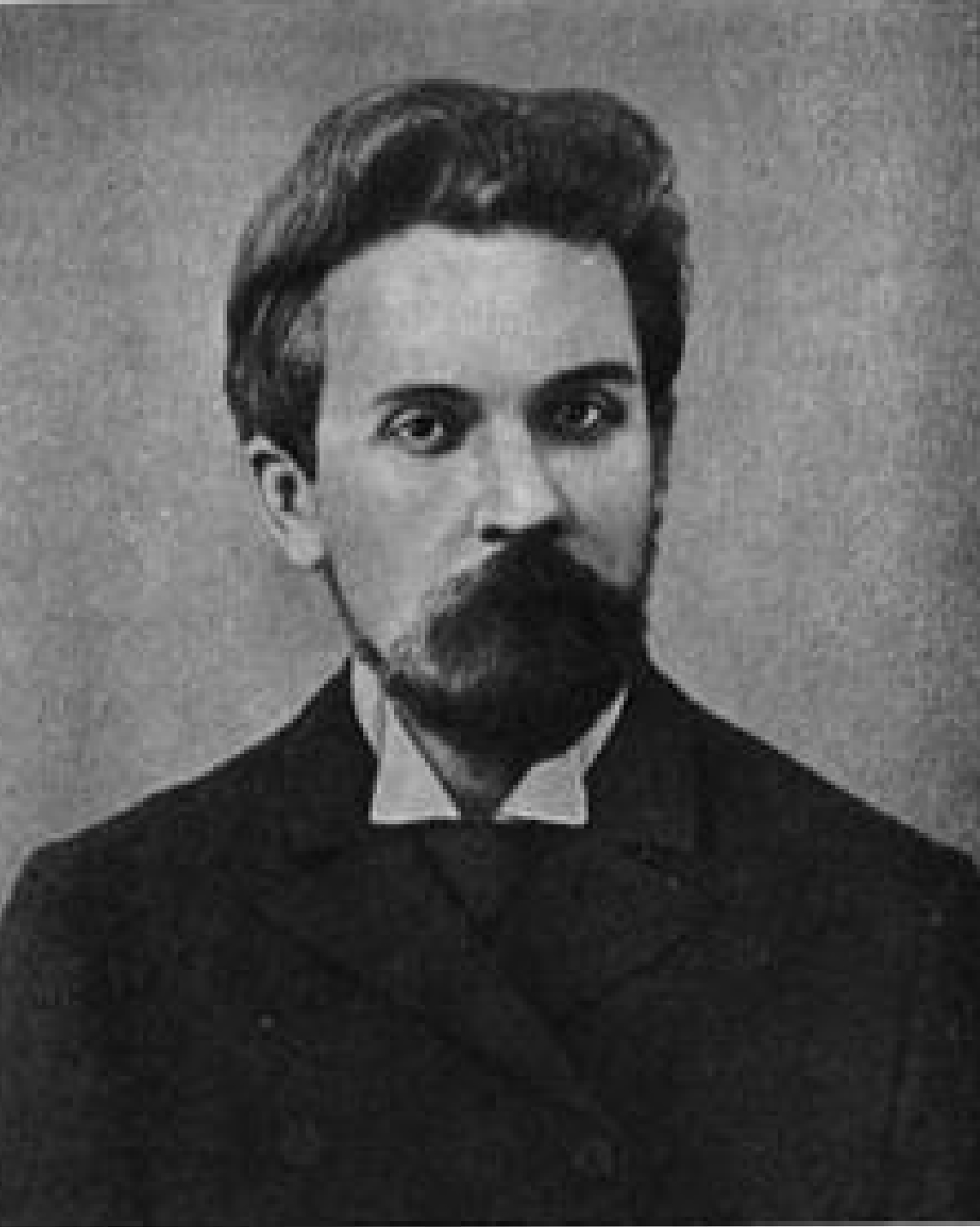}
			&
			\includegraphics[height=3.75cm]{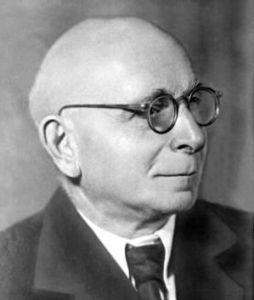}
			\\[4pt]
			
			{\scriptsize D.~I. Mendeleev}
			&
			{\scriptsize A.~A. Markov}
			&
			{\scriptsize V.~A. Markov}
			&
			{\scriptsize S.~Bernstein}
			\\[-1pt]
			
			{\scriptsize (1834--1907)}
			&
			{\scriptsize (1856--1922)}
			&
			{\scriptsize (1871--1897)}
			&
			{\scriptsize (1880--1968)}
			
		\end{tabular*}%
	}
	\caption{Mendeleev and three of the principal figures in the classical
		Bernstein--Markov theory.}
	\label{fig:classical-pioneers}
\end{figure}

Pointwise bounds for higher derivatives were obtained by Duffin and Schaeffer in
1938 \cite{DuffinSchaeffer}. Put
\[
S_n(t):=\sin\bigl(n\arccos t\bigr),\qquad -1\le t\le1.
\]

\begin{theorem}[R.~J. Duffin and A.~C. Schaeffer]
	\label{thm:classical-DuffinSchaeffer}
	Let $P$ be a real polynomial of degree at most $n$, let $1\le k\le n$, and let
	$t\in(-1,1)$. Then
	\begin{equation}\label{eq:classical-DuffinSchaeffer}
		|P^{(k)}(t)|
		\le
		\left(
		\bigl[T_n^{(k)}(t)\bigr]^2+
		\bigl[S_n^{(k)}(t)\bigr]^2
		\right)^{1/2}
		\|P\|_{[-1,1]}.
	\end{equation}
	Consequently,
	\[
	\mathcal B_{n,k}(t)\le \mathcal M_{n,k}(t),
	\qquad
	\mathcal M_{n,k}(t):=
	\left(
	\bigl[T_n^{(k)}(t)\bigr]^2+
	\bigl[S_n^{(k)}(t)\bigr]^2
	\right)^{1/2}.
	\]
\end{theorem}

For $k=1$, the right-hand side in
\eqref{eq:classical-DuffinSchaeffer} reduces to
$n/\sqrt{1-t^2}$, and hence the Duffin--Schaeffer estimate recovers
\cref{thm:classical-Bernstein}. For higher derivatives it improves the global
endpoint bound from \cref{thm:classical-VAMarkov} through a substantial part of
the interval, although the endpoint estimate remains decisive sufficiently close to
$\pm1$. Duffin and Schaeffer also used their pointwise inequality to give a shorter
proof of V.~A. Markov's theorem.

A functional-analytic method for determining the exact pointwise constants was
initiated by Voronovskaja and later developed by Gusev
\cites{Voronovskaja,Gusev}. These methods characterize the extremal problem but do
not generally yield formulas that are easy to evaluate. For degrees two and three,
convex-geometric techniques lead instead to completely explicit expressions
\cite{AMRS}; see also \cite{FSMS}. The formulas needed for comparison are recalled
next.

\begin{theorem}[Exact low-degree Bernstein functions; \cite{AMRS}]
	\label{thm:classical-low-degree}
	For $-1\le t\le1$,
	\begin{equation}\label{eq:classical-B2}
		\mathcal B_2(t)=
		\begin{cases}
			\displaystyle\frac{1}{1-|t|},
			& |t|\le\frac12,\\[2mm]
			4|t|,
			& \frac12\le |t|\le1,
		\end{cases}
	\end{equation}
	
	\begin{equation}\label{eq:classical-B3}
		\mathcal B_3(t)=
		\begin{cases}
			\displaystyle 3(1-4t^2),
			& |t|\le\dfrac{\sqrt7-2}{6},\\[2mm]
			\displaystyle \frac{7\sqrt7+10}{9(1+|t|)},
			& \dfrac{\sqrt7-2}{6}\le |t|
			\le\dfrac{2\sqrt7-1}{9},\\[3mm]
			\displaystyle \frac{16|t|^3}{(9t^2-1)(1-t^2)},
			& \dfrac{2\sqrt7-1}{9}\le |t|
			\le\dfrac{1+2\sqrt7}{9},\\[3mm]
			\displaystyle \frac{7\sqrt7-10}{9(1-|t|)},
			& \dfrac{1+2\sqrt7}{9}\le |t|
			\le\dfrac{\sqrt7+2}{6},\\[3mm]
			\displaystyle 3(4t^2-1),
			& \dfrac{\sqrt7+2}{6}\le |t|\le1,
		\end{cases}
	\end{equation}
	and
	\begin{equation}\label{eq:classical-B32}
		\mathcal B_{3,2}(t)=
		\begin{cases}
			\displaystyle\frac{4}{1-9t^2},
			& |t|\le\frac19,\\[2mm]
			\displaystyle\frac{32}{9(1-|t|)^2},
			& \frac19\le |t|\le\frac13,\\[2mm]
			24|t|,
			& \frac13\le |t|\le1.
		\end{cases}
	\end{equation}
\end{theorem}

The corresponding extremal polynomials, including the changes of extremizer at the
breakpoints in \eqref{eq:classical-B2}--\eqref{eq:classical-B32}, are described in
\cite{AMRS}. 

The problem considered below is different in two important respects: our polynomials
are homogeneous quadratic forms in two variables, and their norm is taken on the
unbalanced planar body $\OO$, rather than on an interval. Nevertheless, the
classical distinction survives unchanged. The pointwise quantity is the best
Euclidean gradient bound at a prescribed point, while its maximum over $\OO$ is the
global Markov constant. The complete extreme-point description recalled in the
preceding section will allow us to compute both quantities exactly.

For $P\in\Ptwo$, write
$\nabla P=(\partial P/\partial x,\partial P/\partial y)$ and use $\|\cdot\|_2$
for the Euclidean norm on $\R^2$.  For \((x,y)\in\R^2\), define
\begin{equation}\label{eq:Bernstein-def}
 \Bern(x,y):=\sup\{\|\nabla P(x,y)\|_2:P\in\Ptwo,\ \normO{P}\le1\}.
\end{equation}
The supremum is a maximum because \(\BO\) is compact.  It is also the operator norm
of the linear map
\[
 (a,b,c)\longmapsto(2ax+by,bx+2cy)
\]
from \((\R^3,\normO{\cdot})\) to \(\ell_2^2\), and therefore \(\Bern\) is continuous.
Moreover,
\begin{equation}\label{eq:Bernstein-symmetry}
 \Bern(tx,ty)=t\Bern(x,y)\quad(t\ge0),
 \qquad
 \Bern(x,y)=\Bern(y,x).
\end{equation}

For a coefficient triple \(P=(a,b,c)\), set
\begin{equation}\label{eq:GP}
 G_P(r):=\left(a+\frac b2r\right)^2+
          \left(\frac b2+cr\right)^2,
 \qquad 0\le r\le1.
\end{equation}
If \(x\ge y\ge0\), \(x>0\), and \(r=y/x\), then
\begin{equation}\label{eq:gradient-normalized}
 \|\nabla P(x,y)\|_2^2=4x^2G_P(r).
\end{equation}

The next lemma is the key simplification: although \(\ext(\BO)\) contains four curves and four pairs of isolated points, only one curve can maximize
the Euclidean gradient.

\begin{lemma}[Dominant extreme curve]\label{lem:dominance}
For every \(0\le r\le1\),
\begin{equation}\label{eq:dominant-curve}
 \max_{P\in\ext(\BO)}G_P(r)
 =\max_{\tau\le u\le1}H(r,u),
\end{equation}
where
\begin{equation}\label{eq:H}
 H(r,u):=
 \left[-1+\frac{1+3u}{2}r\right]^2+
 \left[\frac{1+3u}{2}+\frac{1-3u^2}{2}r\right]^2.
\end{equation}
Thus the right-hand side is the contribution of the curve \(\Gamma_2\).
\end{lemma}

\begin{proof}
Negation does not change \(G_P\), so only the upper curves and the four isolated
points need be considered.  The contributions of \(\Gamma_1,\Gamma_4\), and
\(\Gamma_3\) are, respectively,
\begin{align*}
 H_1(r,s)&=(-1+sr)^2+\bigl(s+(1-s^2)r\bigr)^2,\\
 H_4(r,s)&=(1-s^2+sr)^2+(s-r)^2,\\
 H_3(r,u)&=\left(\frac{1-3u^2}{2}+\frac{1+3u}{2}r\right)^2
          +\left(\frac{1+3u}{2}-r\right)^2.
\end{align*}
A direct simplification gives
\begin{equation}\label{eq:G1-G4}
 H_1(r,s)-H_4(r,s)=s^2(1-r^2)(2-s^2)\ge0,
\end{equation}
and
\begin{equation}\label{eq:G2-G3}
 H(r,u)-H_3(r,u)
 =\frac34(1-r^2)(1-u^2)(3u^2+1)\ge0.
\end{equation}
Hence it remains to compare \(\Gamma_2\) with \(\Gamma_1\) and the isolated points.

The endpoint \(u=1\) of \(\Gamma_2\) gives
\begin{equation}\label{eq:H-endpoint}
 H(r,1)=5r^2-8r+5.
\end{equation}
For \(s\in[0,\tau]\), put \(d=1-r\).  A direct expansion yields
\begin{align}
 H(r,1)-H_1(r,s)
 ={}&(4+s^2-s^4)d^2+2s^2(s^2-s-1)d+s^3(2-s).
 \label{eq:H-G1}
\end{align}
For \(s>0\), the right-hand side is a quadratic in \(d\) with positive leading
coefficient and discriminant
\[
 4s^3(5s-8)<0,
\]
because \(0<s\le\tau<1\).  It is therefore strictly positive for every real
\(d\).  For \(s=0\) it equals \(4d^2\).  Thus
\(H(r,1)\ge H_1(r,s)\) throughout the required parameter rectangle.

The isolated points in \eqref{eq:isolated} contribute, respectively,
\[
 r^2,\qquad 1,\qquad \frac{(1+r)^2}{2},\qquad
 J(r):=\left(1-\frac{\tau r}{2}\right)^2+
       \left(r-\frac\tau2\right)^2.
\]
The first three comparisons follow from
\begin{align*}
 H(r,1)-r^2&=4(r-1)^2+1,\\
 H(r,1)-1&=5\left(r-\frac45\right)^2+\frac45,\\
 H(r,1)-\frac{(1+r)^2}{2}&=\frac92(r-1)^2.
\end{align*}
Finally,
\begin{equation}\label{eq:H-p3}
 H(r,1)-J(r)
 =\frac{5-\sqrt2}{4}
 \bigl[(3+\sqrt2)(r^2+1)-8r\bigr]>0,
\end{equation}
whose last quadratic is positive because its discriminant is
\(20-24\sqrt2<0\).  Combining these comparisons with
\eqref{eq:G1-G4}--\eqref{eq:G2-G3} proves \eqref{eq:dominant-curve}.
\end{proof}

We can now optimize explicitly along \(\Gamma_2\).  For \(0<r\le1\), set
\begin{equation}\label{eq:dru}
 d_r(u):=6r^2u^3-9ru^2+(r^2-2r+3)u+(r^2-r+1).
\end{equation}
Then
\begin{equation}\label{eq:H-derivative}
 \frac{\partial H}{\partial u}(r,u)=\frac32d_r(u).
\end{equation}

\begin{theorem}[Exact Euclidean Bernstein function]\label{thm:Bernstein}
Let \((x,y)\in\OO\), put \(m=\max\{x,y\}\), and, when \(m>0\), set
\[
 r:=\frac{\min\{x,y\}}{m}\in[0,1].
\]
For \(1/2<r\le1\), define
\begin{align}
 p(r)&:=\frac{2r^2-4r-3}{12r^2},
 &q(r)&:=\frac{2r-1}{12r},\label{eq:pq}\\
 \vartheta(r)&:=\frac13\arccos\left(
 \frac{3q(r)}{2p(r)}\sqrt{-\frac3{p(r)}}\right),\label{eq:vartheta}\\
 \omega(r)&:=\frac1{2r}+2\sqrt{-\frac{p(r)}3}
 \cos\left(\vartheta(r)-\frac{2\pi}{3}\right).
 \label{eq:omega-cardano}
\end{align}
Extend \(\omega\) to \([0,1]\) by
\begin{equation}\label{eq:omega}
 \omega(r):=
 \begin{cases}
 1,&0\le r\le\frac12,\\
 \text{the value in \eqref{eq:omega-cardano}},&\frac12<r\le1.
 \end{cases}
\end{equation}
Then \(\Bern(0,0)=0\), and for \(m>0\),
\begin{equation}\label{eq:Bernstein-formula}
\displaystyle
 \Bern(x,y)=2m\sqrt{H\bigl(r,\omega(r)\bigr)}.
\end{equation}
For \(r>1/2\), the number \(\omega(r)\) is the unique zero of \(d_r\) in
\((\tau,1)\).  In particular, \(\tau<\omega(r)<1\).

If \(x\ge y\), equality in \eqref{eq:Bernstein-formula} is attained by the norm-one
polynomial
\begin{equation}\label{eq:Bernstein-extremizer}
 P_r(\xi,\eta)
 =-\xi^2+\bigl(1+3\omega(r)\bigr)\xi\eta
 +\frac{1-3\omega(r)^2}{2}\eta^2,
\end{equation}
and by its negative.  For \(y\ge x\), the first and third coefficients are
interchanged.
\end{theorem}

\begin{proof}
By \cref{thm:geometric-input,lem:dominance} and
\eqref{eq:gradient-normalized}, it remains only to maximize \(H(r,u)\) on
\(\tau\le u\le1\).

For \(r=0\), \(H(0,u)=1+(1+3u)^2/4\), so the maximum is attained at \(u=1\).
Assume \(0<r<1\).  The coefficient sequence of the cubic \(d_r\) has two sign
changes, while
\begin{equation}\label{eq:d-signs}
 d_r(0)>0,
 \qquad d_r(1/r)=r^2-1<0,
 \qquad d_r(u)\longrightarrow+\infty\quad(u\to+\infty).
\end{equation}
Descartes' rule therefore shows that \(d_r\) has exactly two positive zeros: one
in \((0,1/r)\) and one in \((1/r,\infty)\).  Since the product of the three roots
is negative, the remaining root is negative.  Thus \(d_r\) is positive before its
first positive zero, negative between the two positive zeros, and positive again
after the second one.

At the endpoint \(u=1\),
\begin{equation}\label{eq:d-one}
 d_r(1)=4(2r-1)(r-1).
\end{equation}
If \(0<r\le1/2\), then \(d_r(1)\ge0\).  The second positive zero lies after
\(1/r>1\) when \(r<1\).  The sign pattern just established therefore forces the
first positive zero to lie in \([1,1/r)\); at \(r=1/2\) it is exactly \(u=1\).
Thus \(d_r\ge0\) on \([\tau,1]\), so \(H(r,\cdot)\) is increasing there and its
maximum occurs at \(u=1\).

Now let \(1/2<r<1\).  The value at the other endpoint is
\begin{equation}\label{eq:d-tau}
 d_r(\tau)
 =(31\sqrt2-42)r^2+(16\sqrt2-26)r+(3\sqrt2-2)>0.
\end{equation}
Indeed, its leading coefficient is \(31\sqrt2-42>0\), while its discriminant as a
quadratic in \(r\) is \(108-80\sqrt2<0\).  Hence it is positive for every real
\(r\).  Since \(d_r(1)<0\), the first positive zero lies
in \((\tau,1)\), and it is the unique maximizer of \(H(r,\cdot)\) on that interval.
For \(r=1\), the second positive zero is the endpoint \(u=1\), while the first one
remains in \((\tau,1)\) and is again the maximizer.

To write the first positive zero in closed form, substitute \(u=z+1/(2r)\) in
\(d_r(u)=0\).  The equation becomes
\[
 z^3+p(r)z+q(r)=0,
\]
with \(p,q\) as in \eqref{eq:pq}.  The preceding root count shows that all three
roots are real, and \(p(r)<0\).  Cardano's trigonometric formula gives
\[
 z_k=2\sqrt{-\frac{p(r)}3}
 \cos\left(\vartheta(r)-\frac{2k\pi}{3}\right),
 \qquad k=0,1,2.
\]
For \(0\le\vartheta(r)\le\pi/3\), the three displayed values are ordered as
\(z_2\le z_1\le z_0\).  Restoring the common shift preserves this ordering.  Since
the original cubic has one negative root and two positive roots, the middle root
\(k=1\) is the smaller positive root, and hence precisely the one for which
\(\tau<u<1\).  Restoring the shift yields \eqref{eq:omega-cardano}.  The polynomial
in \eqref{eq:Bernstein-extremizer} is the point \(\gamma_2(\omega(r))\) of the
extreme curve, so it has norm one and attains equality.
\end{proof}

\begin{remark}[The phase transition]\label{rem:transition}
For every direction with \(0\le r\le1/2\), the extremizer in
\cref{thm:Bernstein} is the single polynomial
\[
 P_*(\xi,\eta)=-\xi^2+4\xi\eta-\eta^2.
\]
At \(r=1/2\), the maximizer leaves the endpoint \(u=1\) of \(\Gamma_2\) and moves
into the relative interior of the curve.  At the diagonal direction,
\[
 \omega(1)=\frac{3+\sqrt{33}}{12}.
\]
\end{remark}

The homogeneity of \(\Bern\) makes its restriction to the radial boundary especially
informative.  On the half of the boundary with \(x\ge y\), write the point as
\[
 (x,y)=
 \begin{cases}
 (1,r),&0\le r\le\tau,\\[1mm]
 \displaystyle\left(\frac{\sqrt2}{1+r},\frac{\sqrt2r}{1+r}\right),
 &\tau\le r\le1.
 \end{cases}
\]
Hence \eqref{eq:Bernstein-formula} becomes
\begin{equation}\label{eq:boundary-profile}
 \Bern(x,y)=
 \begin{cases}
 2\sqrt{H(r,\omega(r))},&0\le r\le\tau,\\[1mm]
 \displaystyle\frac{2\sqrt2}{1+r}\sqrt{H(r,\omega(r))},&\tau\le r\le1.
 \end{cases}
\end{equation}

\begin{figure}[t]
\centering
\begin{minipage}[c]{.44\textwidth}
\centering
\includegraphics[width=\linewidth]{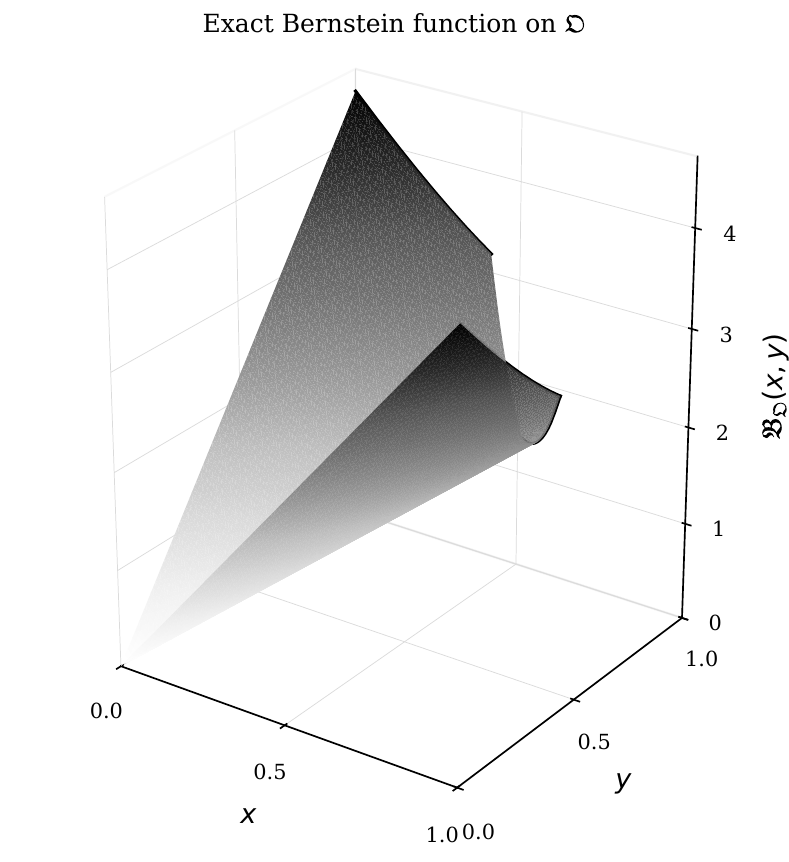}
\end{minipage}\hfill
\begin{minipage}[c]{.54\textwidth}
\centering
\includegraphics[width=\linewidth]{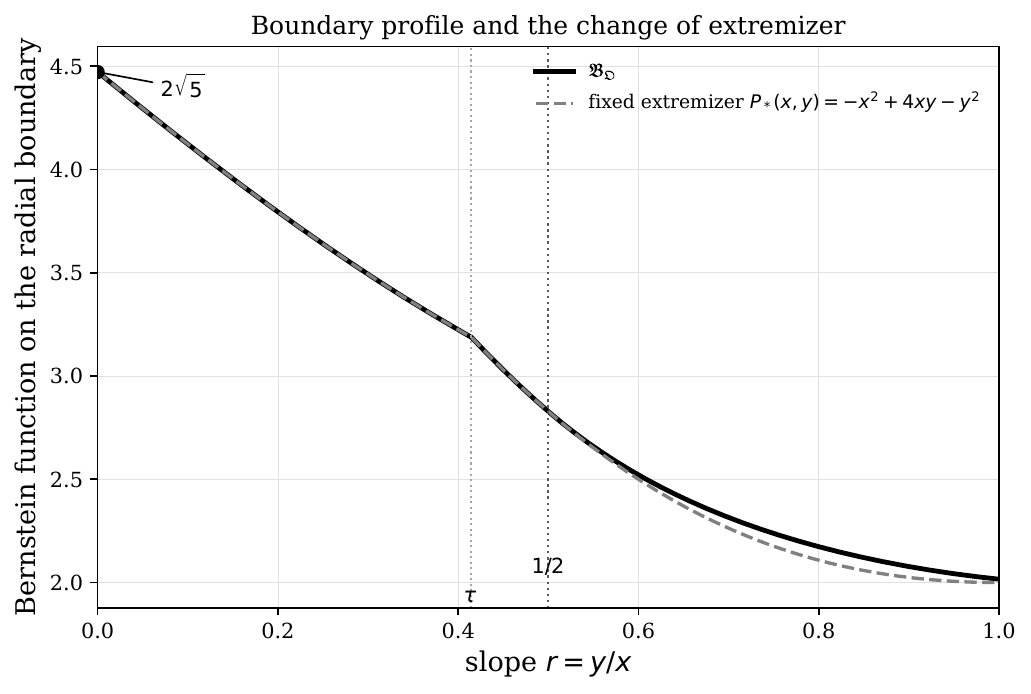}
\end{minipage}
\caption{The exact Bernstein function.  Left: the surface
\((x,y)\mapsto\Bern(x,y)\) over \(\OO\).  Right: its restriction to one half of the
radial boundary.  The dashed curve is the contribution of the fixed polynomial
\(P_*\); it coincides with the optimum until \(r=1/2\), where the active point begins
to move along \(\Gamma_2\).  The line \(r=\tau\) marks the corner \((1,\tau)\).}
\label{fig:Bernstein}
\end{figure}

\begin{figure}[t]
\centering
\includegraphics[width=.72\textwidth]{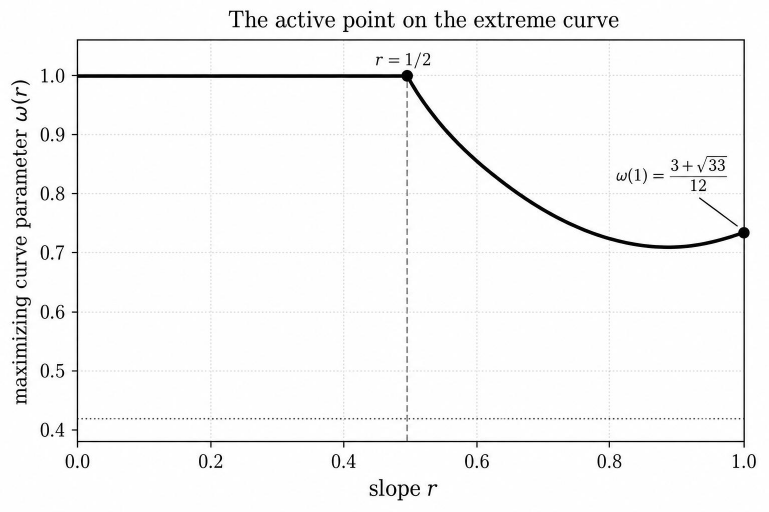}
\caption{The maximizing parameter \(\omega(r)\) on \(\Gamma_2\).  It remains at the
endpoint \(u=1\) for \(r\le1/2\), then follows the explicit Cardano branch in
\eqref{eq:omega-cardano}.}
\label{fig:active-parameter}
\end{figure}

\section{The sharp Euclidean Markov inequality}\label{sec:markov}

Define the Euclidean Markov constant by
\begin{equation}\label{eq:Markov-def}
 \Mark:=\max_{(x,y)\in\OO}\Bern(x,y).
\end{equation}

\begin{theorem}[Sharp Euclidean Markov constant]\label{thm:Markov}
For every \(P\in\Ptwo\),
\begin{equation}\label{eq:Markov-ineq}
\displaystyle
 \max_{(x,y)\in\OO}\|\nabla P(x,y)\|_2
 \le2\sqrt5\,\normO{P}.
\end{equation}
The constant is optimal, so
\begin{equation}\label{eq:Markov-value}
\Mark=2\sqrt5.
\end{equation}
Equality is attained at \((1,0)\) and \((0,1)\) by
\begin{equation}\label{eq:Markov-extremizer}
 P_*(x,y)=-x^2+4xy-y^2
\end{equation}
and by its negative.
\end{theorem}

\begin{proof}
The function \(\Bern\) is convex, being the supremum of the convex functions
\((x,y)\mapsto\|\nabla P(x,y)\|_2\), \(P\in\BO\).  A convex function on a polygon is
bounded above by the largest of its values at the vertices.  The vertices of \(\OO\)
are
\[
 (0,0),\quad(1,0),\quad(1,\tau),\quad(\tau,1),\quad(0,1).
\]
The value at the origin is zero, and symmetry leaves only \((1,0)\) and
\((1,\tau)\) to compare.

At \((1,0)\), \cref{thm:Bernstein} gives \(r=0\), \(\omega(0)=1\), and
\[
 \Bern(1,0)^2=4H(0,1)=20.
\]
At \((1,\tau)\), one has \(\tau<1/2\), so again \(\omega(\tau)=1\), but
\[
 \Bern(1,\tau)^2=4H(\tau,1)
 =4(28-18\sqrt2)<20.
\]
Thus the maximum is \(2\sqrt5\).  The endpoint \(u=1\) of \(\Gamma_2\) is exactly
the polynomial \(P_*\), which proves the equality statement.
\end{proof}

\begin{remark}
The global extremizer is not selected by a separate optimization.  It is the same
endpoint that controls every point whose slope satisfies \(r\le1/2\).  The exact comparison above shows that the loss of radial scale at the non-axial
vertices outweighs any directional gain, so the global maximum occurs at the two
axial vertices.
\end{remark}

\section{The relative polarization constant}\label{sec:polarization}

We begin with the classical Banach-space statement in its full degree-$n$ form.
If $X$ is a real or complex Banach space, $P:X\to\mathbb K$ is a continuous
$n$-homogeneous polynomial, and $\check P:X^n\to\mathbb K$ is its unique symmetric
$n$-linear polar, put
\[
 \|P\|_X:=\sup_{\|x\|\le1}|P(x)|,
 \qquad
 \|\check P\|_X:=\sup_{\|x_j\|\le1}
       |\check P(x_1,\ldots,x_n)|.
\]

\begin{theorem}[Martin's polarization inequality, 1932]
\label{thm:Martin-classical}
For every such $P$,
\begin{equation}\label{eq:Martin-classical}
 \|P\|_X\le \|\check P\|_X
 \le \frac{n^n}{n!}\,\|P\|_X.
\end{equation}
The factor $n^n/n!$ is generally optimal: equality is attained for
$X=\ell_1^n$ and the Nachbin polynomial
$N_n(x_1,\ldots,x_n)=x_1\cdots x_n$.
\end{theorem}

The first inequality is immediate from $P(x)=\check P(x,\ldots,x)$; the content of
Martin's theorem is the dimension-free reverse estimate.  In degree two it becomes
\begin{equation}\label{eq:Martin-quadratic}
 \|P\|_X\le \|\check P\|_X\le2\|P\|_X.
\end{equation}
Thus the usual quadratic polarization constant
\[
 \mathbf c(2,X):=\sup_{P\ne0}
 \frac{\|\check P\|_X}{\|P\|_X}
\]
satisfies $1\le\mathbf c(2,X)\le2$.  See
\cites{Martin,Dineen,GJMMMS} for the theorem, its historical context, and further
polarization results.

The sector $\OO$ is not balanced and hence is not the unit ball of a normed space.
Martin's theorem therefore does not apply literally.  We instead compare the two
supremum norms generated by the same body; the resulting constant is relative to
$\OO$, not a Banach-space invariant.

Let \(P(x,y)=ax^2+bxy+cy^2\).  Its associated symmetric bilinear form is
\begin{equation}\label{eq:polar}
 \check P((x_1,y_1),(x_2,y_2))
 =ax_1x_2+\frac b2(x_1y_2+x_2y_1)+cy_1y_2.
\end{equation}
We equip it with the relative bilinear norm
\begin{equation}\label{eq:bilnorm-def}
 \bilnorm{\check P}:=
 \sup\{|\check P(z,w)|:z,w\in\OO\}.
\end{equation}
Since diagonal pairs are allowed, \(\normO{P}\le\bilnorm{\check P}\).

\begin{definition}\label{def:polarization}
The quadratic polarization constant relative to \(\OO\) is
\begin{equation}\label{eq:polarization-def}
 \mathbf c_2(\OO):=
 \sup_{P\ne0}\frac{\bilnorm{\check P}}{\normO{P}}.
\end{equation}
\end{definition}

The adjective ``relative'' emphasizes that \eqref{eq:polarization-def} compares
the polynomial and bilinear supremum norms induced by \(\OO\), rather than defining
a Banach-space invariant.  It should also not be confused with the asymptotic
polarization constant of a Banach space.

\begin{lemma}[Bilinear norm at the vertices]\label{lem:bilinear-norm}
For \(P=(a,b,c)\),
\begin{equation}\label{eq:bilinear-norm}
\begin{aligned}
 \bilnorm{\check P}=\max\Bigl\{&|a|,\ |c|,\ \frac{|b|}{2},\\
 &\left|a+\frac{\tau b}{2}\right|,
  \left|\tau a+\frac b2\right|,
  \left|\frac b2+\tau c\right|,
  \left|\frac{\tau b}{2}+c\right|,\\
 &|a+\tau b+\tau^2c|,
  |\tau^2a+\tau b+c|,
  \left|\tau(a+c)+(1-\tau)b\right|\Bigr\}.
\end{aligned}
\end{equation}
\end{lemma}

\begin{proof}
The body \(\OO\) is the convex hull of
\[
 0,\quad e_1=(1,0),\quad e_2=(0,1),\quad
 v_1=(1,\tau),\quad v_2=(\tau,1).
\]
For fixed one variable, the absolute value of \(\check P\) is a convex function of the
other variable, and hence its maximum over a polygon occurs at a vertex.  Applying
this successively in the two variables reduces the norm to unordered pairs of the four
nonzero vertices.  Their ten values are exactly those displayed in
\eqref{eq:bilinear-norm}; here
\((1+\tau^2)/2=1-\tau\).
\end{proof}

\begin{theorem}[Exact polarization constant]\label{thm:polarization}
The relative quadratic polarization constant of the octagonal sector is
\begin{equation}\label{eq:polarization-value}
\mathbf c_2(\OO)=2.
\end{equation}
The constant is attained by \(P_*=-x^2+4xy-y^2\) and by its negative.
\end{theorem}

\begin{proof}
The map \(P\mapsto\bilnorm{\check P}\) is continuous and convex on \(\BO\), so
\cref{thm:geometric-input} reduces its maximum to the extreme set.

For \(P_*\), the term \(|b|/2\) in \eqref{eq:bilinear-norm} equals \(2\).  Since
\(P_*\in\Gamma_2\subset\ext(\BO)\), this proves
\(\mathbf c_2(\OO)\ge2\).

We prove the reverse inequality.  The isolated points satisfy
\begin{equation}\label{eq:isolated-bilnorm}
 \bilnorm{\check P_{p_j}}=1,
 \qquad j=0,1,2,3.
\end{equation}
Consider first \(\Gamma_1\), written as
\(P_s=(-1,2s,1-s^2)\), \(0\le s\le\tau\).  The four diagonal vertex values in
\eqref{eq:bilinear-norm} have absolute value at most one.  The six cross values are bounded by
\begin{align*}
 s&\le\tau,\\
 |-1+\tau s|&\le1,\\
 |s-\tau|&\le\tau,\\
 s+\tau(1-s^2)&\le\tau+\tau(1-\tau^2)<1,\\
 1-s^2+\tau s&=1+\frac{\tau^2}{4}-\left(s-\frac\tau2\right)^2
               \le1+\frac{\tau^2}{4},\\
 2(1-\tau)s-\tau s^2&\le\tau.
\end{align*}
For the fourth and sixth expressions, the corresponding functions are increasing on
\([0,\tau]\); the displayed bounds are their values at \(s=\tau\).  The fifth
bound is attained at \(s=\tau/2\).  Therefore
\begin{equation}\label{eq:G1-bilnorm}
 \sup_{P\in\Gamma_1}\bilnorm{\check P}
 =1+\frac{\tau^2}{4}<2.
\end{equation}
The same conclusion holds on \(\Gamma_4\) by symmetry.

Now let \(P_u\in\Gamma_2\), so
\[
 P_u=\left(-1,1+3u,\frac{1-3u^2}{2}\right),
 \qquad \tau\le u\le1.
\]
Again the diagonal values are at most one.  For the cross values,
\begin{align*}
 \frac{1+3u}{2}&\le2,\\
 \left|-1+\tau\frac{1+3u}{2}\right|&\le1,\\
 -\tau+\frac{1+3u}{2}&\le2-\tau<2,\\
 \frac{1+3u}{2}+\tau\frac{1-3u^2}{2}&\le2-\tau<2.
\end{align*}
The fourth expression is increasing on \([\tau,1]\).  The remaining vertex cross
term
\[
 g(u):=\tau\frac{1+3u}{2}+\frac{1-3u^2}{2}
\]
is decreasing on this interval and satisfies
\(g(\tau)=1/\sqrt2\) and \(g(1)=-\tau^2\); hence
\(|g(u)|\le1/\sqrt2<2\).  For the final cross pair \((v_1,v_2)\), substitution gives
\[
 h(u):=1-\frac{3\tau}{2}+3(1-\tau)u-\frac{3\tau}{2}u^2.
\]
Since \(h'(u)=3(1-\tau-\tau u)>0\) on \([\tau,1]\),
\[
 1=h(\tau)\le h(u)\le h(1)=4-6\tau<2.
\]
Therefore \(\bilnorm{\check P_u}\le2\) on \(\Gamma_2\), with equality at \(u=1\)
through the term \(|b|/2\).  The curve \(\Gamma_3\) is symmetric.  Together with
\eqref{eq:isolated-bilnorm}--\eqref{eq:G1-bilnorm}, this proves the upper bound and
hence \eqref{eq:polarization-value}.
\end{proof}

\begin{corollary}[Sharp coordinate bounds]\label{cor:coefficient-box}
For every \(P(x,y)=ax^2+bxy+cy^2\),
\begin{equation}\label{eq:coefficient-bound}
\displaystyle
 \max\left\{|a|,|c|,\frac{|b|}{4}\right\}\le\normO{P}.
\end{equation}
All three bounds are sharp.  Equivalently, the smallest axis-parallel box containing
\(\BO\) is
\begin{equation}\label{eq:coordinate-box}
 [-1,1]\times[-4,4]\times[-1,1].
\end{equation}
\end{corollary}

\begin{proof}
Evaluation at \((1,0)\) and \((0,1)\) gives
\(|a|,|c|\le\normO{P}\).  Since \(e_1,e_2\in\OO\),
\[
 \frac{|b|}{2}=|\check P(e_1,e_2)|
 \le\bilnorm{\check P}
 \le2\normO{P}
\]
by \cref{thm:polarization}.  The first two bounds are attained by \(x^2\) and \(y^2\),
and the third by \(P_*\).
\end{proof}

\begin{corollary}[Sharp persistence of Martin's quadratic factor]\label{cor:Martin}
For every quadratic form,
\begin{equation}\label{eq:Martin-sector}
 \bilnorm{\check P}\le2\normO{P},
\end{equation}
and the constant \(2\) cannot be improved.
\end{corollary}

\begin{proof}
This is exactly the inequality and optimality statement of
\cref{thm:polarization}.
\end{proof}

Although \(\OO\) is not a Banach-space unit ball, \cref{cor:Martin} shows that
Martin's numerical degree-two factor persists sharply on this nonsymmetric sector.
This is a genuine geometric fact, not a formal consequence of
\cref{thm:Martin-classical}: on the simplex the corresponding relative constant is
$3$, whereas on the square it is $3/2$; see \cref{tab:comparison}.

\section{Coefficient and unconditional inequalities}\label{sec:coefficients}

\subsection[Exact coefficient norms]{Exact coefficient \texorpdfstring{$\ell_q$}{lq}-norm constants}

Let
\[
 P(x,y)=ax^2+bxy+cy^2.
\]
For $1\le q<\infty$, define
\begin{equation}\label{eq:coefficient-q}
 |P|_q:=\bigl(|a|^q+|b|^q+|c|^q\bigr)^{1/q},
\end{equation}
and put $|P|_\infty:=\max\{|a|,|b|,|c|\}$.  These are norms on the
coefficient space $\mathbb R^3$; they are different from the polynomial supremum
norm $\normO{P}$.

\begin{definition}\label{def:Cq}
For $1\le q\le\infty$, the coefficient constant of $\OO$ at exponent $q$ is
\begin{equation}\label{eq:Cq-def}
 \mathcal C_q(\OO):=\sup_{P\ne0}\frac{|P|_q}{\normO{P}}.
\end{equation}
\end{definition}

Because $\Ptwo$ is finite-dimensional, $\mathcal C_q(\OO)<\infty$ for every
$q$.  The point is to determine the exact value, including its extremizers.

\begin{theorem}[Exact coefficient-norm constants]\label{thm:coefficient-q}
For every $1\le q<\infty$,
\begin{equation}\label{eq:Cq-value}
 \displaystyle
 \mathcal C_q(\OO)=(2+4^q)^{1/q},
\end{equation}
and
\begin{equation}\label{eq:Cinf-value}
\mathcal C_\infty(\OO)=4.
\end{equation}
Every constant is attained by
\[
 P_*(x,y)=-x^2+4xy-y^2
\]
and by its negative.
\end{theorem}

\begin{proof}
By homogeneity, suppose $\normO{P}=1$.  The sharp coordinate bounds in
\cref{cor:coefficient-box} give
\[
 |a|\le1,\qquad |b|\le4,\qquad |c|\le1.
\]
Consequently,
\[
 |P|_q\le(1+4^q+1)^{1/q}=(2+4^q)^{1/q}
 \quad(1\le q<\infty),
\]
and $|P|_\infty\le4$.  Since $P_*$ has norm one and coefficient vector
$(-1,4,-1)$, equality holds simultaneously for every $q$.
\end{proof}

\subsection[The exponent 4/3 and Bohnenblust--Hille theory]{The exponent \texorpdfstring{$4/3$}{4/3} and Bohnenblust--Hille theory}

We recall both classical formulations in order to make the scope of the next
corollary unambiguous.  Let $\mathbb K\in\{\mathbb R,\mathbb C\}$.  The
\emph{multilinear Bohnenblust--Hille constant} in degree $m$ is the least number
$B^{\mathrm{mult}}_{\mathbb K,m}$ such that
\begin{equation}\label{eq:multilinear-BH}
 \left(
  \sum_{i_1,\ldots,i_m=1}^{N}
  |T(e_{i_1},\ldots,e_{i_m})|^{\frac{2m}{m+1}}
 \right)^{\frac{m+1}{2m}}
 \le B^{\mathrm{mult}}_{\mathbb K,m}\,\|T\|
\end{equation}
holds for every positive integer $N$ and every $m$-linear form
$T:(\ell_\infty^N)^m\to\mathbb K$.  The crucial feature is that the same constant
works for all dimensions $N$; equivalently, the estimate extends to continuous
$m$-linear forms on $c_0$.  When $m=2$, \eqref{eq:multilinear-BH} is
Littlewood's $4/3$ inequality, and in the real case its optimal constant is
$\sqrt2$ \cites{DinizEtAl,Littlewood}.

The \emph{polynomial Bohnenblust--Hille constant} in degree $m$ is defined
analogously.  It is the least number $B^{\mathrm{pol}}_{\mathbb K,m}$ such that,
for every $N$ and every polynomial
\[
 Q(z)=\sum_{|\alpha|=m}a_\alpha z^\alpha
 \quad(z\in\mathbb K^N),
\]
one has
\begin{equation}\label{eq:classical-BH}
 \left(
  \sum_{|\alpha|=m}|a_\alpha|^{\frac{2m}{m+1}}
 \right)^{\frac{m+1}{2m}}
 \le B^{\mathrm{pol}}_{\mathbb K,m}
 \sup_{\|z\|_\infty\le1}|Q(z)|.
\end{equation}
Again the constant is uniform in $N$---equivalently, it controls homogeneous
polynomials on $c_0$---and the exponent $2m/(m+1)$ is optimal in this
dimension-free sense \cites{BohnenblustHille,DefantEtAl}.  Constants obtained
after fixing $N$ are useful finite-dimensional variants, but they are not the same
objects as the dimension-free constants.

The next consequence uses the quadratic exponent $4/3$, but the norm is taken on
the fixed body $\OO\subset\mathbb R^2$, not on $B_{\ell_\infty^N}$ uniformly in
$N$.  We therefore call it a Bohnenblust--Hille-\emph{type} inequality, not a
Bohnenblust--Hille inequality.

\begin{corollary}[Sharp $4/3$-coefficient inequality]\label{cor:BH-type}
For every quadratic form $P(x,y)=ax^2+bxy+cy^2$,
\begin{equation}\label{eq:BH-type-sector}
\displaystyle
 \bigl(|a|^{4/3}+|b|^{4/3}+|c|^{4/3}\bigr)^{3/4}
 \le \bigl(2+2^{8/3}\bigr)^{3/4}\normO{P}.
\end{equation}
The factor is optimal and equals $\mathcal C_{4/3}(\OO)$.  Equality is attained by
$P_*$ and $-P_*$.
\end{corollary}

\begin{proof}
Apply \cref{thm:coefficient-q} with $q=4/3$ and use
$4^{4/3}=2^{8/3}$.
\end{proof}

\begin{remark}[Why $\mathcal C_{4/3}(\OO)$ is not a Bohnenblust--Hille constant]
\label{rem:not-BH}
There are three separate reasons.  First, the number of variables is fixed at two,
whereas \eqref{eq:multilinear-BH} and \eqref{eq:classical-BH} require one constant
uniformly for all $N$.  Second, $\OO$ is not the unit ball of $\ell_\infty^2$ (or of
any normed space, since it is unbalanced).  Third, no critical-exponent phenomenon
is involved here: \cref{thm:coefficient-q} computes the exact constant for every
$1\le q\le\infty$.  The phrase ``of Bohnenblust--Hille type'' refers only to the
appearance of the classical quadratic exponent $4/3$.

For comparison, the fixed two-variable real polynomial constant associated with the
classical $\ell_\infty^2$ norm at exponent $4/3$ is approximately $1.837373$
\cite{JMMS}, while the factor in \eqref{eq:BH-type-sector} is approximately
$4.911899$.  The difference reflects the strong cancellation allowed by the
first-quadrant norm: the norm-one polynomial $P_*$ has coefficient vector
$(-1,4,-1)$.
\end{remark}

\subsection{The unconditional constant of the canonical monomial basis}

Let
\[
 \mathcal B_{\mathrm{can}}=(m_1,m_2,m_3):=(x^2,xy,y^2)
\]
be the canonical ordered basis of $\Ptwo$.  Its unconditional constant is the least
$C\ge1$ such that
\begin{equation}\label{eq:unconditional-definition}
 \left\|\sum_{j\in A}\varepsilon_j a_j m_j\right\|_{\OO}
 \le C\left\|\sum_{j=1}^3a_jm_j\right\|_{\OO}
\end{equation}
for every coefficient vector $(a_1,a_2,a_3)\in\mathbb R^3$, every subset
$A\subset\{1,2,3\}$, and every choice of signs
$\varepsilon_j\in\{-1,1\}$.  We denote it by $\Unc$.

For $P=ax^2+bxy+cy^2$, define its coefficient modulus polynomial by
\begin{equation}\label{eq:modulus-polynomial}
 P^\#(x,y):=|a|x^2+|b|xy+|c|y^2.
\end{equation}
Because $\OO$ lies in the positive quadrant, the standard modulus reduction is
particularly transparent.

\begin{lemma}[Modulus characterization]\label{lem:unconditional-modulus}
The canonical unconditional constant satisfies
\begin{equation}\label{eq:unconditional-modulus}
 \Unc=\sup_{P\ne0}\frac{\normO{P^\#}}{\normO{P}}.
\end{equation}
\end{lemma}

\begin{proof}
For $(x,y)\in\OO$, all three monomials $x^2,xy,y^2$ are nonnegative.  Therefore,
for every subset $A$ and every choice of signs,
\[
 \left|\sum_{j\in A}\varepsilon_j a_j m_j(x,y)\right|
 \le\sum_{j=1}^3|a_j|m_j(x,y)=P^\#(x,y).
\]
Taking suprema gives the upper bound in \eqref{eq:unconditional-definition} with
$C$ equal to the right-hand side of \eqref{eq:unconditional-modulus}.  Conversely,
choose $A=\{1,2,3\}$ and choose $\varepsilon_j$ so that
$\varepsilon_j a_j=|a_j|$.  The sign-changed polynomial is then exactly $P^\#$,
which proves equality.  This is the finite-dimensional modulus formulation of the
canonical unconditional constant used in \cite{GMSU}.
\end{proof}

\begin{theorem}[Exact unconditional constant]\label{thm:unconditional}
The canonical monomial basis of $(\Ptwo,\normO{\cdot})$ has unconditional constant
\begin{equation}\label{eq:unconditional-value}
\Unc=3.
\end{equation}
The constant is attained by $P_*=-x^2+4xy-y^2$.
\end{theorem}

\begin{proof}
Normalize $\normO{P}=1$.  By \cref{cor:coefficient-box},
$|a|\le1$, $|b|\le4$, and $|c|\le1$.  Hence, pointwise on $\OO$,
\[
 P^\#(x,y)\le x^2+4xy+y^2=(x+y)^2+2xy.
\]
Since $x+y\le\sqrt2$ and $xy\le(x+y)^2/4$,
\begin{equation}\label{eq:modulus-upper}
 P^\#(x,y)
 \le\frac32(x+y)^2\le3.
\end{equation}
Thus \cref{lem:unconditional-modulus} gives $\Unc\le3$.

For $P_*$, the modulus polynomial is
\[
 P_*^\#(x,y)=x^2+4xy+y^2.
\]
At $(x,y)=(1/\sqrt2,1/\sqrt2)\in\OO$ it takes the value $3$; hence
$\normO{P_*^\#}=3$.  Since $\normO{P_*}=1$, equality follows.  Equivalently, the
norm-one polynomial $x^2-4xy+y^2=-P_*$ becomes a norm-three polynomial when only
the sign of its mixed term is changed.
\end{proof}

\begin{remark}\label{rem:C1-versus-unconditional}
The coefficient $\ell_1$ constant and the unconditional constant are not the same:
\[
 \mathcal C_1(\OO)=6,
 \qquad
 \Unc=3.
\]
The first quantity is the largest coefficient sum
$|a|+|b|+|c|$ on $\BO$; the second is the largest \emph{polynomial norm}
$\normO{P^\#}$.  The distinction is essential on $\OO$, because the point
$(1,1)$, at which a positive quadratic would realize its coefficient sum, does not
belong to the sector.
\end{remark}

The unconditional constant also gives a body-relative Bohr-radius statement.
For $r>0$, write $r\OO:=\{rz:z\in\OO\}$ and, with
$\|Q\|_{r\OO}:=\sup_{z\in r\OO}|Q(z)|$, define
\begin{equation}\label{eq:bohr-radius-definition}
 R_{\mathrm{can}}(\OO):=
 \sup\{r>0:\|P^\#\|_{r\OO}\le1\text{ whenever }\|P\|_{\OO}\le1\}.
\end{equation}

\begin{corollary}[Canonical body-relative Bohr radius]\label{cor:bohr-radius}
\begin{equation}\label{eq:bohr-radius-value}
R_{\mathrm{can}}(\OO)=\frac1{\sqrt3}.
\end{equation}
\end{corollary}

\begin{proof}
Quadratic homogeneity gives
$\|P^\#\|_{r\OO}=r^2\|P^\#\|_{\OO}$.  By
\cref{thm:unconditional}, this is at most $3r^2\|P\|_{\OO}$, so
$r=1/\sqrt3$ is admissible.  The polynomial $P_*$ shows that no larger radius is
possible.
\end{proof}

\begin{corollary}[A universal extremizer for the global inequalities]
\label{cor:common-extremizer}
The global constants satisfy
\begin{equation}\label{eq:relation}
\Mark=2\sqrt{1+\mathbf c_2(\OO)^2}=2\sqrt5.
\end{equation}
Moreover, $P_*=-x^2+4xy-y^2$ is simultaneously extremal for the Markov
inequality, the relative polarization inequality, every coefficient-norm inequality in
\cref{thm:coefficient-q}, the $4/3$-coefficient inequality in
\cref{cor:BH-type}, the canonical unconditional constant, and the radius in
\cref{cor:bohr-radius}.
\end{corollary}

\begin{proof}
The identity follows from \cref{thm:Markov,thm:polarization}.  The equality
statements in \cref{thm:Markov,thm:polarization,thm:coefficient-q,thm:unconditional}
identify the same polynomial $P_*$; the radius statement then follows from
quadratic homogeneity.
\end{proof}

\subsection{Comparison with related exact constants}

For context, \cref{tab:comparison} places the octagonal result between the two
endpoint bodies
\[
 \Delta\subset\OO\subset\square,
 \qquad
 \Delta:=\{(x,y)\in[0,1]^2:x+y\le1\},
 \quad \square:=[0,1]^2.
\]
For a body $K$ in this list, $C_{xy}(K)$ is the best factor in
$|b|\le C_{xy}(K)\|P\|_K$, and $R_{\mathrm{can}}(K)$ is the canonical
body-relative Bohr radius.  Quadratic homogeneity gives
$R_{\mathrm{can}}(K)=\mathbf u_{\mathrm{can}}(K)^{-1/2}$.  The endpoint Markov,
polarization, mixed-coefficient, and unconditional constants are taken from
\cites{GMSS,GMSU,MRS}.  The extremizers
\[
 x^2-6xy+y^2,\qquad -x^2+4xy-y^2,\qquad x^2-3xy+y^2
\]
for the simplex, octagonal sector, and square, respectively, simultaneously attain
the two pure-coefficient bounds and the sharp mixed-coefficient bound.  Hence, if
$m_K=C_{xy}(K)$, then
\begin{equation}\label{eq:endpoint-Cq-family}
 \mathcal C_q(K)=(2+m_K^q)^{1/q}\quad(1\le q<\infty),
 \qquad \mathcal C_\infty(K)=m_K.
\end{equation}
Thus the second panel of the table records the complete coefficient scale at four
representative exponents, not merely the value at $4/3$.

\begin{table}[htbp]
\centering
\caption{Sharp constants for three nested unbalanced planar bodies.  Panel (a)
contains global, polarization, and canonical-basis quantities; panel (b) records
selected values from the exact coefficient scale in
\eqref{eq:endpoint-Cq-family}.}
\label{tab:comparison}
\footnotesize
\begingroup
\renewcommand{\arraystretch}{1.15}
\setlength{\tabcolsep}{3.7pt}
\begin{tabular}{@{}lccccc@{}}
\toprule
\multicolumn{6}{c}{\textit{(a) Global, polarization, and basis constants}}\\
\cmidrule(lr){1-6}
Body $K$ & $\mathfrak M_K$ & $\mathbf c_2(K)$
& $\mathbf u_{\mathrm{can}}(K)$ & $R_{\mathrm{can}}(K)$
& $C_{xy}(K)$\\
\midrule
Simplex
& $2\sqrt{10}$ & $3$ & $2$ & $1/\sqrt2$ & $6$\\
Octagonal sector
& $2\sqrt5$ & $2$ & $3$ & $1/\sqrt3$ & $4$\\
Square
& $\sqrt{13}$ & $3/2$ & $5$ & $1/\sqrt5$ & $3$\\
\bottomrule
\end{tabular}

\medskip

\begin{tabular}{@{}lcccc@{}}
\toprule
\multicolumn{5}{c}{\textit{(b) Exact coefficient-norm constants}}\\
\cmidrule(lr){1-5}
Body $K$ & $\mathcal C_1(K)$ & $\mathcal C_{4/3}(K)$
& $\mathcal C_2(K)$ & $\mathcal C_\infty(K)$\\
\midrule
Simplex
& $8$ & $\bigl(2+6^{4/3}\bigr)^{3/4}$ & $\sqrt{38}$ & $6$\\
Octagonal sector
& $6$ & $\bigl(2+4^{4/3}\bigr)^{3/4}$ & $3\sqrt2$ & $4$\\
Square
& $5$ & $\bigl(2+3^{4/3}\bigr)^{3/4}$ & $\sqrt{11}$ & $3$\\
\bottomrule
\end{tabular}
\endgroup
\end{table}

The canonical unconditional constant depends on the polynomial norm together with
the chosen coordinate monomials; even an isometry of the underlying plane need not
preserve this coordinate structure.  For general $1<p<\infty$, $p\ne2$, the
canonical unconditional constant of quadratic forms on $B_{\ell_p^2}$ is described
in \cite{GMSU} as the maximum of an explicit one-variable function.  Several
special values simplify to closed forms.  Together with the endpoint cases and
Kim's octagonal formula, they give the broader comparison in
\cref{tab:unconditional-balanced}.

For the full octagonal family, put
\[
 \mathcal O_w:=\{(x,y)\in\mathbb R^2:|x|\le1,\ |y|\le1,
 |x|+|y|\le1+w\},\qquad 0<w\le\tau.
\]
Kim proved \cite{KimOctUnc} that
\begin{equation}\label{eq:Kim-unconditional-family}
 \mathbf u_{\mathrm{can}}(\mathcal O_w)
 =\frac{1+w^2+\sqrt{2(1+w^4)}}{(1+w)^2}.
\end{equation}
Since $\widetilde{\OO}=\mathcal O_\tau$, substitution of
$\tau=\sqrt2-1$ gives
\[
 \mathbf u_{\mathrm{can}}(\widetilde{\OO})
 =2-\sqrt2+\sqrt3(\sqrt2-1).
\]

\begin{table}[htbp]
\centering
\caption{Selected exact canonical unconditional constants for quadratic forms on
balanced planar norming bodies.  The $\ell_p^2$ values are from \cite{GMSU}; the
regular-octagon value follows from \cite{KimOctUnc}.}
\label{tab:unconditional-balanced}
\footnotesize
\begingroup
\renewcommand{\arraystretch}{1.16}
\setlength{\tabcolsep}{5.2pt}
\begin{tabular}{@{}lcc@{}}
\toprule
Norming body & Exact constant & Decimal value\\
\midrule
$B_{\ell_1^2}$
& $\dfrac{1+\sqrt2}{2}$ & $1.2071\ldots$\\[1mm]
Full regular octagon $\widetilde{\OO}$
& $2-\sqrt2+\sqrt3(\sqrt2-1)$ & $1.3032\ldots$\\[1mm]
$B_{\ell_{3/2}^2}$
& $\dfrac12\bigl(3+5\sqrt2+2\sqrt{2+10\sqrt2}\bigr)^{1/3}$
& $1.3130\ldots$\\[1mm]
$B_{\ell_2^2}$ & $\sqrt2$ & $1.4142\ldots$\\[1mm]
$B_{\ell_3^2}$
& $\bigl(\dfrac12+\dfrac5{\sqrt2}\bigr)^{1/3}$ & $1.5921\ldots$\\[1mm]
$B_{\ell_4^2}$ & $\sqrt3$ & $1.7321\ldots$\\[1mm]
$B_{\ell_5^2}$
& $\bigl(\dfrac12+\dfrac{29}{\sqrt2}\bigr)^{1/5}$ & $1.8385\ldots$\\[1mm]
$B_{\ell_6^2}$ & $2^{1/6}5^{1/3}$ & $1.9194\ldots$\\[1mm]
$B_{\ell_8^2}$ & $17^{1/4}$ & $2.0305\ldots$\\[1mm]
$B_{\ell_\infty^2}$ & $1+\sqrt2$ & $2.4142\ldots$\\
\bottomrule
\end{tabular}
\endgroup
\end{table}

\begin{figure}[htbp]
\centering
\includegraphics[width=.98\textwidth]{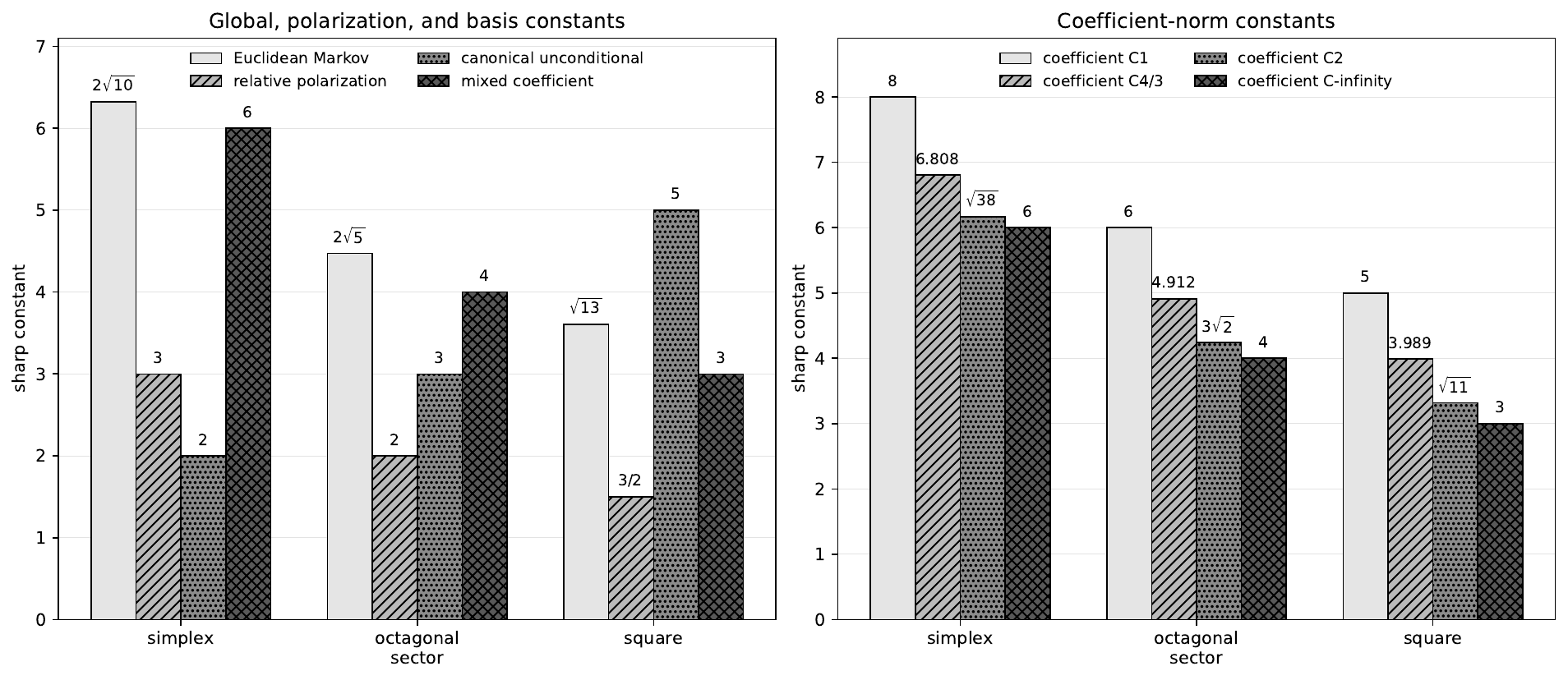}
\caption{Two complementary comparisons for the simplex, the octagonal sector,
and the square.  The left panel contains global, polarization, basis, and mixed-
coefficient constants; the right panel displays four values from the exact
coefficient scale.  In particular, $\mathcal C_{4/3}(K)$ is a fixed-space
coefficient constant of Bohnenblust--Hille type, not a classical dimension-free
Bohnenblust--Hille constant.  The figure does not assert monotonicity for arbitrary
inclusions of convex bodies.}
\label{fig:constants-comparison}
\end{figure}
%\FloatBarrier
\section{Final remarks}

The companion geometric description contains four extreme curves, their negatives,
and four isolated pairs, yet the Euclidean Bernstein problem is governed by only one
of those curves.  Its endpoint produces the global extremizer $P_*$, while beyond
slope $1/2$ the pointwise extremizer moves through the interior of the same curve.
Thus one explicit geometric family explains both a local phase transition and the
sharp global derivative constant.

The coefficient applications reveal a complementary phenomenon.  The same endpoint
polynomial simultaneously saturates the polarization inequality, every coefficient
$\ell_q$ comparison, and the unconditional inequality.  These constants measure
different aspects of the norm: for example,
\[
 \mathcal C_1(\OO)=6
 \qquad\hbox{but}\qquad
 \Unc=3.
\]
The first quantity measures the coefficient sum, whereas the second measures the
supremum norm of the modulus polynomial.  Their difference is a concrete expression
of the fact that $\normO{\cdot}$ is a non-absolute polynomial norm.

The $4/3$ estimate should be interpreted with equal care.  Its exponent is the
quadratic Bohnenblust--Hille exponent, but its constant is attached to one fixed
three-dimensional polynomial space.  Classical Bohnenblust--Hille constants are
uniform over the number of variables and hence encode an infinite-dimensional
phenomenon.  The present result is instead an exact finite-dimensional coefficient
inequality of Bohnenblust--Hille type.  Making this distinction explicit also shows
why the complete scale $\mathcal C_q(\OO)$, rather than only the value at $4/3$, is
the natural result in the current setting.

The method extends in principle whenever a low-dimensional polynomial unit ball has
finitely many parametrized extreme families.  The main work is then not the
Krein--Milman reduction itself, but the comparison of competing families and the
exact solution of the surviving one-dimensional optimizations.  For polygonal or
piecewise-algebraic bodies, this may lead simultaneously to derivative,
polarization, unconditional, and coefficient inequalities governed by a common
extreme polynomial.

\subsection*{Acknowledgements}
Manwook Han thanks the Universidad Complutense de Madrid for its hospitality during
his visit in November 2024.

\subsection*{Funding}
Han and Kim were supported by the National Research Foundation of Korea (NRF) grant
funded by the Korea government (MSIT) [NRF-2020R1C1C1A01012267].

\begingroup
\setlength{\emergencystretch}{2em}

\endgroup

\end{document}